\documentclass[12pt]{amsart}
\usepackage{a4wide}
\usepackage{eucal}
\usepackage{amsthm,color}
\usepackage{enumerate}

\numberwithin{equation}{section}
\newtheorem{cor}{Corollary}[section]
\newtheorem{theorem}[cor]{Theorem}
\newtheorem{prop}[cor]{Proposition}
\newtheorem{proposition}[cor]{Proposition}
\newtheorem{lemma}[cor]{Lemma}
\newtheorem{question}[cor]{Question}
\theoremstyle{definition}

\newtheorem{definition}[cor]{Definition}
\theoremstyle{remark}
\newtheorem{remark}[cor]{Remark}
\newtheorem{example}[cor]{Example}
\newtheorem{conjecture}[cor]{Conjecture}

\newcommand{\cF}{\mathcal{F}}

\newcommand{\cZ}{\mathcal{Z}}
\def\Z{{\mathcal Z}}

\newcommand{\fP}{\mathfrak{P}}

\newcommand{\cz}{{\mathbb C}}
\newcommand{\rz}{{\mathbb R}}

\let\mR\RM

\let\mC\CM

\def\one{1}

\newcommand{\SU}{\operatorname{SU}}

\newcommand{\vol}{\mathop{\mathrm{vol}}}
\newcommand{\dvol}{\mathop{\mathrm{dvol}}}
\newcommand{\dime}{\mathop{\mathrm{dim}}}
\newcommand{\kernel}{\mathop{\mathrm{ker}}}
\newcommand{\ind}{\mathop{\mathrm{ind}}}
\newcommand{\grad}{\mathop{\mathrm{grad}}}

\def\Dirac{{\mathchoice{D\kern-8pt\slash}{D\kern-8pt\slash}{{\scriptstyle D\kern-5.5pt\slash}}{{\scriptscriptstyle D\kern-5pt\slash}}}}

\def\e{\varepsilon}

\def\P{\mathrm{P}}

\let\al\alpha

\let\la\lambda
\let\si\sigma
\def\a{\alpha}

\def\res{|}
\def\nz{\n^\Z}

\def\Hol{{\rm Hol}}
\def\Ric{{\rm Ric}}
\def\End{{\rm End}}
\def\tr{{\rm tr}}
\def\Id{{\rm Id}}
\def\Scal{{\rm Scal}}
\def\vol{{\rm vol}}
\newcommand\eref[1]{(\ref{#1})}

\newcommand{\be}{\begin{equation}}
\newcommand{\ee}{\end{equation}}
\def\beq{\begin{eqnarray*}}
\def\eeq{\end{eqnarray*}}
\def\p{\psi}
\def\P{\Psi}
\def\.{{\cdot}}
\def\gz{g^\Z}
\def\n{\nabla}
\def\nm{\nabla^g}

\long\def\willbeignored#1{}
\newcommand{\lamin}{\lambda_{\rm min}^+}
\newcommand{\mS}{\mathbb{S}}
\let\om\omega
\let\<\langle
\let\>\rangle
\let\na\nabla
\let\pa\partial
\let\Si\Sigma
\let\ti\tilde
\def\sign{\mathrm{sign}}
\def\id{\mathrm{id}}

\def\gfl{g_{\rm flat}}

\def\cR{\mathcal{R}}

\def\cRinv{\cR_{U,\gfl}^{\rm inv}}
\def\cRi{\cR^{\rm inv}}

\begin{document}
\title{The Cauchy problems for Einstein metrics and parallel spinors}

\author{Bernd Ammann} 
\address{Bernd Ammann \\ Fakult\"at f\"ur  Mathematik \\ 
Universit\"at Regensburg \\
93040 Regensburg \\  
Germany}
\email{bernd.ammann@mathematik.uni-regensburg.de}

\author{Andrei Moroianu}
\address{Andrei Moroianu \\ Universit\'e de Versailles-St Quentin \\
Laboratoire de Math\'ematiques \\ UMR 8100 du CNRS\\
45 avenue des \'Etats-Unis\\
78035 Versailles, France }
\email{am@math.polytechnique.fr}

\author{Sergiu Moroianu}
\address{Sergiu Moroianu \\ Institutul de Matematic\u{a} al Academiei Rom\^{a}ne\\
P.O. Box 1-764\\RO-014700
Bu\-cha\-rest, Romania}
\email{moroianu@alum.mit.edu}
\date{\today}

\begin{abstract}
The restriction of a parallel spinor on some spin manifold $\Z$ to a hypersurface $M\subset \Z$ is a generalized Killing spinor 
on $M$. We show, conversely, that in the real analytic category, every spin manifold $(M,g)$ carrying a generalized Killing spinor 
$\psi$ can be isometrically embedded as a hypersurface in a spin manifold carrying a parallel spinor whose restriction to $M$ is $\psi$.
We also answer negatively the corresponding question in the smooth category.
\end{abstract}
\maketitle

\section{Introduction}

This paper aims to solve the problem of extending a spinor from a hypersurface to a parallel spinor on the total space. This problem is related to that of extending a Riemannian metric on a hypersurface to an Einstein metric on the total space, since parallel spinors can only exist over Ricci-flat manifolds. 

\subsection*{The Cauchy problem for Einstein metrics}
In the Lorentzian setting, Ricci-flat or more generally Einstein metrics
form the central objects of general relativity. Given a space-like hypersurface, a Riemannian metric, and
a symmetric tensor which plays the role of the second fundamental form, there always exists a local extension
to a Lorentzian Einstein metric \cite{cb}, \cite{dt}, provided that the local conditions given by the
contracted Gauss and Codazzi-Mainardi equations are satisfied. One crucial step in the proof
is the reduction to an evolution equation which
is (weakly) hyperbolic due to the signature of the metric. 
The corresponding equations in the Riemannian setting are  
(weakly) elliptic and no general local existence results are available.

In fact, if $(M,g)$ is any hypersurface of an Einstein manifold $(\Z,g^\Z)$, 
then the Weingarten tensor~$W$ is a symmetric endomorphism field on $M$ which satisfies certain
constraints (see \eqref{e1}--\eqref{e2} below, which are contractions of the
Gauss and Codazzi-Mainardi equations). Conversely, one can ask the following question:
\begin{quote}{\em
(Q1): If $W$ is a symmetric endomorphism field on $M$ which satisfies the system \eqref{e1}--\eqref{e2},
does there exist a isometric embedding of $M$ into a (Riemannian) Einstein manifold
$(\Z^{n+1},g^\Z)$ with Weingarten tensor $W$? Is $g^Z$ unique near $M$ up to isometry?}
\end{quote}

The uniqueness part is known to have a positive answer 
by recent results of Biquard \cite[Thm.\ 4]{biq} and Anderson-Herzlich \cite{ah}. 
The existence was settled in a paper by Koiso \cite{ko} in the real analytic setting. As we were unaware of that paper, in a previous draft of this work we had proved in detail that the answer to the existence part of the above Cauchy problem 
is positive in the analytic setting (Theorem \ref{rf}). We review the proof in Section 2 and show that the answer is negative, in general, in
the smooth setting (Proposition \ref{non-an}).

Let us also mention that DeTurck \cite{dt2} analyzed in the Riemannian setting the somewhat related
problem of finding a metric with prescribed nonsingular Ricci tensor.

\subsection*{Extension of generalized Killing spinors to parallel spinors}
Our main focus in this paper is the extension problem for spinors. 
In order to introduce it, we must recall some basic facts about restrictions of spin bundles to
hypersurfaces.
If $\Z$ is a Riemannian spin manifold,
any oriented hypersurface $M\subset\Z$ inherits a spin structure and
it is well-known that the restriction to $M$ of the complex
spin bundle $\Sigma\Z$ if $n$ is even (resp.\ $\Sigma ^+\Z$ if $n$ is odd) is canonically
isomorphic to the complex spin bundle $\Sigma M$
(cf.\ \cite{bgm}). If $W$ denotes the Weingarten tensor of $M$, the spin
covariant derivatives $\nz$ on $\Sigma \Z$ and $\nm$ on $\Sigma M$ are related by (\cite[Eq.\ (8.1)]{bgm})
\begin{align}
\label{spin}(\nz_X\P)|_M=\nm_X(\P|_M)-\tfrac12 W(X)\.(\P|_M),&&
\forall\  X\in TM,
\end{align}
for all spinors (resp.\ half-spinors for $n$ odd) $\P$ on $\Z$. 
We thus see that if $\P$ is a parallel spinor on
$\Z$, its restriction $\p$ to any hypersurface $M$ is a {\em generalized
  Killing spinor} on $M$, i.e. it satisfies the equation
\begin{align}
\label{gks}\nm_X\p=\tfrac12W(X)\.\p,&&\forall\ X\in TM,
\end{align}
and the symmetric tensor $W$, called the stress-energy tensor of $\p$,
is just the Weingarten tensor of the hypersurface $M$. It is
natural to ask whether the converse holds:

\begin{quote}{\em
(Q2): If $\p$ is a generalized Killing spinor on $M^n$, does there
  exist an isometric embedding of $M$ into a spin manifold
  $(\Z^{n+1},g^\Z)$ carrying a parallel spinor $\P$ whose restriction
  to $M$ is $\p$?}
\end{quote}

This question is the Cauchy problem for metrics with parallel
spinors asked in~\cite{bgm}.

The answer is known to be positive in several special cases: if the
stress-energy tensor~$W$ of $\p$ is the identity \cite{ba}, if $W$ is
parallel \cite{morel03} and if $W$ is a Codazzi tensor \cite{bgm}. Even earlier,
Friedrich \cite{friedrich:98} had worked out the 2-dimensional case $n+1=2+1$, which 
is also covered by \cite[Thm. 8.1]{bgm} since on surfaces the stress-energy of a generalized Killing 
spinor is automatically a Codazzi tensor. Some related embedding results were also obtained by
Kim \cite{kim}, Lawn--Roth \cite{lr10} and Morel \cite{morel05}.
The common feature of each of 
these cases is that one can actually construct in an explicit way the
``ambient" metric $g^\Z$ on the product $(-\e,\e)\times M$.

Our aim is to show that the same is true more generally, under the sole
additional assumption that $(M,g)$ and $W$ are analytic.

\begin{theorem} \label{m1}
Let $\p$ be a spinor field on an analytic spin
  manifold $(M^n,g)$, and $W$ an analytic field of symmetric endomorphisms of $TM$. Assume
that $\p$ is a generalized Killing spinor with respect to $W$, i.e.\ it satisfies \eqref{gks}.
Then there exists a unique metric $g^\Z$ of the form $g^\Z=dt^2+g_t$, with $g_0=g$, on a sufficiently 
small neighborhood $\Z$ of $\{0\}\times M$ inside $\rz\times M$ such that
  $(\Z,g^\Z)$, endowed with the spin structure induced from $M$,
  carries a parallel spinor $\P$ whose restriction to $M$ is $\p$.
\end{theorem}
In particular, the solution $g^\Z$ must be Ricci-flat.
Einstein manifolds are analytic but of course hypersurfaces can lose this structure so our
hypothesis is restrictive. Note that Einstein metrics with smooth initial 
data can be constructed for small time as constant sectional curvature metrics 
when the second fundamental form is a Codazzi tensor, see~\cite[Thm.\ 8.1]{bgm}. 
In particular in dimensions $1+1$ and $2+1$ Theorem \ref{m1} remains 
valid in the smooth category since the tensor $W$ associated to a generalized Killing spinor 
is automatically a Codazzi tensor in dimensions~$1$ and~$2$. 

The situation changes drastically in higher dimensions for smooth (instead of analytic)
generalized Killing spinors.
What we can still achieve then is to solve the Einstein equation (and the parallel spinor
equation) in Taylor series near the initial hypersurface. More precisely, starting from a smooth hypersurface $(M,g)$ with prescribed
Weingarten tensor $W$ we prove that there exist formal Einstein metrics $g^\Z$ such that
$W$ is the second fundamental form at $t=0$, i.e., we solve the Einstein equation modulo
rapidly vanishing errors. Guided by the analytic and the low dimensional ($n=1$ or $n=2$)
cases, one could be tempted to guess that actual germs of Einstein metrics do exist for any smooth initial data.
However this turns out to be false. Counterexamples were found very recently in some particular cases 
in dimensions 3 and 7 by Bryant \cite{bryant10}. 
We give a general procedure to construct counterexamples in all dimensions in Section \ref{counterex}.

Note that several particular instances of Theorem \ref{m1} have been proved in recent years, 
based on the characterization of generalized Killing spinors in terms of exterior forms in low dimensions.
Indeed, in dimensions 5, 6 and 7, generalized Killing spinors are equivalent to
so-called {\em hypo}, {\em half-flat} and {\em co-calibrated} $G_2$ structures respectively. In \cite{h03}
Hitchin proved that the cases $6+1$ and $7+1$ can be solved up to the local existence of a certain gradient flow. Later on, Conti and Salamon \cite{cs06}, \cite{cs07} treated the cases $5+1$, $6+1$ and $7+1$
in the analytical setting, cf.\ also \cite{c08} for further developments.

A construction related to the Cauchy problem for Einstein metrics has been studied starting with the work of Fefferman-Graham \cite{fg}
concerning asymptotically 
hyperbolic Poincar\'e-Einstein metrics. The starting hypersurface $(M^n,g_0)$
is then at infinite distance from the 
manifold $\Z=(0,\e)\times M$, the metric $g^\Z$ being conformal to a  metric $\bar{g}$ of class $C^{n-1}$ on the manifold with boundary $\overline\Z=[0,\e)\times M$:
\begin{align*}
g^\Z=x^{-2} \bar{g}, && \bar{g}=dx^2+ g_x
\end{align*}
such that 
the conformal factor $x$ is precisely the distance function to the boundary ${x=0}$
with respect to $\bar{g}$. The metric is required to be 
Einstein of negative curvature up to an error term which vanishes with
all derivatives at infinity. Such a metric always
exists; when~$n$ is odd, it is smooth down to $x=0$ and its Taylor series at infinity is determined 
by the initial metric $g_0$ and the symmetric transverse traceless tensor $g_n$
appearing as coefficient of $x^n$ in $g_x$, while in even dimensions some logarithmic terms must be allowed, more precisely $g_x$ is smooth as a function of $x$ and $x^n\log x$.

Let us stress that existence results of Einstein metrics with prescribed first 
fundamental form and Weingarten tensor clearly cannot hold globally in general
(Example \ref{exng}).

\subsection*{Counterexamples in the smooth setting}

In the second part of the paper (Section \ref{counterex}) we apply the existence results from the analytic setting 
to prove nonexistence of solutions for certain smooth initial data in any dimension at least $3$. 

The argument goes along the lines of works of the first author and his collaborators on the Yamabe 
problem and the mass endomorphism. We consider the functional
\[{\cF}(\phi):=\frac{\langle D_0\phi,\phi\rangle_{L^2}}{\|D_0\phi\|^2_{L^{2n/(n+1)}}}\]
defined on the $C^1$ spinor fields $\phi$ on a compact connected Riemannian spin manifold $(M,g_0)$ which are 
not in the kernel of the Dirac operator $D_0$. If the infimum of the 
lowest positive eigenvalue of the Dirac operator in the 
volume-normalized conformal class of $g_0$ is strictly 
lower than the corresponding eigenvalue for the standard sphere 
(Condition \eqref{ineq.strict} below), this functional attains its 
supremum in a spinor $\psi_0$ of regularity $C^{2,\alpha}$. 
Moreover, $\psi_0$ is smooth outside its zero set. 

To construct $g_0$ satisfying Condition \eqref{ineq.strict} we fix $p
\in M$ and we look at metrics on $M$ which are flat near $p$. If the topological index of $M$ 
vanishes in $KO^{-n}(pt)$, then for generic such metrics
the associated Dirac operator is invertible. The mass endomorphism at $p$ is defined as the 
constant term in the asymptotic expansion of the Green kernel of $D$ near $p$. Again for generic metrics,
this mass endomorphism is non-zero, which by a 
result of~\cite{ammann.humbert.morel:06}
ensures the technical Condition~\eqref{ineq.strict} for generic metrics 
which are flat near 
$p$. By construction this class of metrics contains metrics which are not conformally flat on some open subset of $M$,
i.e., whose Schouten tensor (in dimension $3$), resp.\ Weyl curvature (in higher dimensions) is nonzero on some
open set. We assume $g_0$ was chosen with these properties.

We return now to the spinor $\psi_0$ maximizing the functional $\cF$.
The Euler-Lagrange equation of $\cF$ at $\psi_0$ can be reinterpreted as follows: the 
Dirac operator with respect to the conformal metric $g:=|\psi_0|^{4/(n-1)}g_0$ admits an 
eigenspinor of constant length $1$, $\psi:=\tfrac{\psi_0}{|\psi_0|}$. 

If the dimension $n$ equals 
$3$, by algebraic reasons this spinor field must be a generalized Killing spinor 
with stress-energy tensor $W$ of constant trace.

The metric $g$ is defined on the complement $M^*$ of the zero set of $\psi_0$. This set 
is open, connected and dense in $M$ (Lemmata \ref{dense} and \ref{connected}). Recall that $g_0$ was chosen such that its
Schouten tensor vanishes identically on an open set of $M$ and is nonzero on another 
open set. Then the same remains true on $M^*$, and therefore on $M^*$ there exists 
no analytic metric in the conformal
class of $g_0$. In particular, the metric $g=|\psi_0|^{4/(n-1)}g_0$ cannot be analytic.

Assuming now that Theorem \ref{rf} continues to hold for smooth initial data, we could apply it to
$(M^*,g,W)$ to get an embedding in a Ricci-flat (hence analytic) Riemannian manifold $(\Z,\gz)$, 
with second fundamental form $W$. Since the trace of $W$ is constant by construction, 
$M$ would have constant mean curvature, which would imply that it were analytic 
(Lemma \ref{cmc.analytic}), contradicting the non-analyticity proved above.

The above construction actually yields counterexamples to the Cauchy problem for 
Ricci-flat metrics in the smooth setting in any dimension $n\geq 3$, by taking products
with flat spaces, see Lemma \ref{product}.

\subsection*{Acknowledgements} It is a pleasure to thank Olivier Biquard, Gilles Carron, Mattias Dahl, Paul Gauduchon,
Colin Guillarmou, Christophe Margerin, Yann Rollin and Jean-Marc Schlenker for helpful discussions.
We thank the DFG-Graduiertenkolleg GRK 1692 Regensburg for its support.
AM was partially supported by the contract ANR-10-BLAN 0105 
``Aspects Conformes de la G\'eom\'etrie'' and by
the LEA ``MathMode''. 
SM was partially supported by the contract PN-II-RU-TE-2011-3-0053 and by
the LEA ``MathMode''.
He thanks the CMLS at the Ecole Polytechnique for its hospitality during the writing of this paper.

\section{The Cauchy problem for Einstein metrics}

Let $(\Z,\gz)$ be an oriented Riemannian manifold of dimension $n+1$, and
$M$ an oriented hypersurface with induced Riemannian metric $g:=g^\Z|_M$. 
We start by fixing some notations. Denote by $\nz$ and $\nm$ 
the Levi-Civita covariant derivatives on $(\Z,g^\Z)$ and $(M,g)$, by $\nu$ the
unit normal vector field along $M$ compatible with the orientations,
and by $W\in\End(TM)$ the Weingarten tensor defined by
\begin{align}\label{w} \nz_X\nu=-W(X),&&\forall\ X\in TM.
\end{align}
Using the normal geodesics issued from $M$, the metric on~$\Z$ can be
expressed in a neighborhood $\Z_0$ of $M$ as 
$g^\Z=dt^2+g_t$, where $t$ is the distance function to $M$ and $g_t$
is a family of Riemannian metrics on $M$ with $g_0=g$ (cf. \cite{bgm}). The vector
field $\nu$ extends to $\Z_0$ as
$\nu=\partial/\partial t$ and \eqref{w} defines a symmetric endomorphism on
$\Z_0$ which can be viewed as a family $W_t$ of endomorphisms of $M$,
symmetric with respect to $g_t$, and satisfying (cf.\ \cite[Equation (4.1)]{bgm}):
\begin{align}\label{4.1}g_t(W_t(X),Y)=-\tfrac12\dot g_t(X,Y),&&\forall\ X,Y\in TM.
\end{align}
By \cite[Equations (4.5)--(4.8)]{bgm}, the Ricci tensor and the
scalar curvature of $\Z$ satisfy for every vectors $X,Y\in TM$
\begin{align}
\label{4.5}&\Ric^\Z(\nu,\nu)=\tr(W_t^2)-\tfrac12\tr_{g_t}(\ddot g_t),\\
\label{4.6}&\Ric^\Z(\nu,X)=d\tr(W_t)(X)+\delta^{g_t}(W)(X),\\
\label{4.7}&\Ric^\Z(X,Y)=\Ric^{g_t}(X,Y)+2g_t(W_tX,W_tY)+\tfrac12\tr(W_t)\dot
g_t(X,Y)-\tfrac12 \ddot g_t(X,Y),\\
\label{4.8}&\Scal^\Z=\Scal^{g_t}+3\tr(W_t^2)-\tr^2(W_t)-\tr_{g_t}(\ddot g_t).
\end{align}
where in \eqref{4.6} the divergence operator $\delta ^g:\End(TM)\to T^*M$ is defined
in a local $g$-orthonormal basis $\{e_i\}$ of $TM$ by
\be\label{div}\delta ^g(A)(X)=-\sum_{i=1}^ng((\nm_{e_i}A)(e_i),X).\ee
Using \eqref{4.5} and \eqref{4.8} we get
\be\label{4.9}-2\Ric^\Z(\nu,\nu)+\Scal^\Z=\Scal^{g_t}+\tr(W_t^2)-\tr^2(W_t).
\ee

Assume now that the metric $g^\Z$ is Einstein with scalar curvature $(n+1)\lambda$, i.e.
$\Ric^\Z=\lambda g^\Z$.
Evaluating
\eqref{4.6} and \eqref{4.9} at $t=0$ yields
\begin{align}\label{e1}&d\tr(W)+\delta ^{g}W=0,\\
\label{e2} &\Scal^{g}+\tr(W^2)-\tr^2(W)=(n-1)\lambda.
\end{align}
If $g_t:\End(TM)\to T^*M\otimes T^*M$ is the isomorphism defined by
$g_t(A)(X,Y):=g_t(A(X),Y)$ and
$g_t^{-1}:T^*M\otimes T^*M\to\End(TM)$ denotes its inverse, then taking
\eqref{4.5} into account, \eqref{4.7} reads
\be\label{4.10}\ddot g_t=2\Ric^{g_t}+\dot
g_t(g_t^{-1}(\dot g_t)\cdot,\cdot)-\tr(g_t^{-1}(\dot g_t))\dot g_t-2\lambda g_t.
\ee

Using the Cauchy-Kowalewskaya theorem, Koiso proved the following existence and unique continuation result for 
Einstein metrics starting from an analytic metric and an analytic stress-energy tensor satisfying the above constraints.

\begin{theorem}[\cite{ko}] \label{rf} 
Let $(M^n,g)$ be an analytic Riemannian
  manifold and let $W$ be an analytic symmetric endomorphism field on
  $M$ satisfying \eqref{e1} and \eqref{e2}. Then for
  $\e>0$, there exists a unique analytic germ near $\{0\}\times M$ of an Einstein
  metric $g^\Z$ with scalar curvature $(n+1)\lambda$ of the form $g^\Z=dt^2+g_t$ on 
  $\Z:=\rz\times M$, with $g_0=g$, whose Weingarten tensor at $t=0$ is $W$.
\end{theorem}
\begin{proof}[Sketch of proof]
In equation \eqref{4.10} the only term involving partial derivatives
of the metric $g_t$ along $M$ is $\Ric^{g_t}$, which is an analytic
expression in $g_t$ and its first and second order derivatives along $M$
which does not involve any derivative with respect to $t$. 

The second order Cauchy-Kowalewskaya theorem (see e.g.\ \cite{dieudonne})
shows that for every $x\in M$ there
exists a neighborhood $V_x\ni x$ and some $\e_x>0$ such
that the Cauchy problem \eqref{4.10} with initial data
\begin{align*}
g_0=g,&&
\dot g_0=-2W
\end{align*}
has a unique analytic solution $g_t$ on $(-\e_x,\e_x)\times V_x$.
Using the uniqueness of solutions for systems of linear ODE's, one can then prove that $g^\Z=dt^2+g_t$ is Einstein with scalar curvature $(n+1)\lambda$. By uniqueness, these metrics patch up to a global metric near $M\times\{0\}$.
\end{proof}

As direct consequences of Theorem \ref{rf}, one obtains the following embedding results for analytic metrics and conformal structures:
\begin{cor}\label{umbilic}
Let $(M^n,g)$ be an analytic Riemannian manifold of constant scalar curvature.
Then for every $\lambda\leq\frac{\Scal^g}{n-1}$, $M$ can be isometrically embedded as a totally umbilical hypersurface in an Einstein manifold 
$(\Z^{n+1},g^\Z)$ with Ricci constant $\lambda$, i.e., $\Ric^\Z=\lambda g^\Z$.
\end{cor}
\begin{proof}
The tensor $W:=\alpha \id$ satisfies Equations \eqref{e1}, \eqref{e2} for $\alpha=\sqrt{\tfrac{\Scal^g}{n(n-1)}-\tfrac\lambda{n}}$
\end{proof}
A conformal structure $c$ on a manifold $M$ is called {\em analytic} if there exists an analytic atlas 
on $M$ and an analytic metric $g$ in $c$ (see Definition \ref{def.conf.ana} below).
\begin{cor}
Let $(M^n,c)$ be a compact analytic conformal manifold of Yamabe invariant $Y(M,c)$. 
Then $M$ can be conformally embedded
as a totally umbilical hypersurface in an Einstein
manifold $(\Z^{n+1},g^\Z)$ with $\Ric^\Z=\lambda g^\Z$ for 
every $\lambda$ with $\sign(\lambda)\leq \sign(Y(M,c))$.
\end{cor}
\begin{proof} Let $g_0$ be some analytic metric in $c$. 
Using the solution to the Yamabe problem for compact manifolds we get a unit volume metric
 $g=u^{4/(n-2)}g_0\in c$ 
with constant scalar curvature $\Scal^g=Y(M,c)$. The function $u$ satisfies a linear elliptic second order differential equation 
(the conformal Laplacian)
with analytic coefficients, so $g$ is analytic. The result now follows from the previous corollary, after a suitable
constant rescaling of $g$.
\end{proof}

\subsection*{The Cauchy problem for smooth initial data}
It was proven recently by Biquard \cite[Thm.\ 4]{biq} and Anderson-Herzlich \cite{ah} that
even in the $C^\infty$ setting, given a hypersurface $M\subset\Z$, a Riemannian metric on~$M$
and a field of symmetric endomorphisms $W$, there exists  (up to diffeomorphisms preserving the hypersurface) 
at most one Einstein metric on~$Z$ with Weingarten tensor $W$ along $M$.

The small-time existence however is known to fail in general for elliptic (even linear) Cauchy problems with
$C^\infty$ initial data. In the particular case of the Cauchy problem for Einstein metrics, we first remark
that in small dimensions the short-time existence is always guaranteed by the construction of an explicit solution
in the smooth (and actually even $C^3$) setting.

Indeed,  in dimension $1$ we can embed any curve $(M,g)$ in a constant curvature surface
with prescribed extrinsic curvature function (identified with the scalar Weingarten tensor) $W$. 
In this case, the constraint equations are empty, 
and the metric is explicitly given by \cite[Theorem 7.2]{bgm}.

Similarly, in dimension $n=2$,
the $C^3$ initial value problem can always be solved for small time:
\begin{prop}
Let $M$ be a surface with $C^3$ Riemannian metric $g$, and let $W$ be a $C^3$
symmetric field of endomorphisms on $M$ satisfying \eqref{e1} and \eqref{e2}
for some $\lambda\in\rz$. Then there exists a  metric $g^\Z$ of constant sectional curvature  $\kappa=\lambda/2$
  on a neighborhood of $\{0\}\times M$ inside $\Z:=\rz \times M$ of the form $g^\Z=dt^2+g_t$, with $g_0=g$,
  whose Weingarten tensor at $t=0$ is $W$.
\end{prop}
\begin{proof}
Direct application of \cite[Theorem 7.2]{bgm}. Namely, in dimension
$2$ the hypotheses \eqref{e1}, \eqref{e2} are equivalent to \cite[Eq.\ (7.3)]{bgm}
resp.\ \cite[Eq.\ (7.4)]{bgm} with $\kappa=\lambda/2$. It follows, at least in the smooth case,
that $g_t$ can be constructed explicitly in terms of $g$ and $W$ such that 
$g^\Z$ has constant sectional curvature $\kappa$. It remains to note that the proof
of \cite[Theorem 7.2]{bgm} remains valid when $g$ and $W$ are of class $C^3$.
\end{proof}

In higher dimensions $n\ge 3$ the situation changes dramatically. In some particular cases
one can show that the analyticity of the initial data is not only sufficient but also necessary:

\begin{prop}\label{non-an}
A Riemannian manifold $(M^n,g)$ of constant scalar curvature can be isometrically embedded in an Einstein manifold 
$(\Z^{n+1},g^\Z)$ with Weingarten tensor $W=\alpha \id$ along $M$ if and only if $g$ is analytic.
\end{prop}
\begin{proof}
The tensor $W$ satisfies Equations \eqref{e1}, \eqref{e2} for $\Scal^g=(n-1)(\lambda+n\alpha)$. The ``if" part thus follows from 
Theorem \ref{rf}. Conversely, if such an embedding exists, then $(M,g)$ is a constant mean curvature 
hypersurface in $(\Z,g^{\Z})$, so $g$ has to be analytic by Lemma \ref{cmc.analytic} below.
\end{proof}

Note that a metric with constant scalar curvature is automatically analytic in dimensions~1 and~2. Examples of
non-analytic constant scalar curvature metrics in dimensions at least~3 can be easily constructed: 
perturb the round metric on~$\mS^n$ to a metric $g$ which is non-conformally flat on some open set and 
conformally flat on some other open set and choose a constant scalar curvature metric in the conformal class of $g$  using the solution
of the Yamabe problem.

\subsection*{Formal solution in the smooth case}

The previous arguments show that without the hypothesis that $g$ and $W$ are analytic,
the nonlinear PDE system \eqref{4.10} has no solution in general. 
However, it is rather evident from \eqref{4.10} that the full Taylor series of $g^\Z$ is recursively determined by its first two coefficients, which are $g$ and $W$.
Let $\dot{C}^\infty(\Z)$ denote the space of tensors vanishing at $M$ together with all their derivatives. By the Borel lemma (see e.g. \cite{gg}), there exists a metric $g^\Z$ such that its Ricci tensor
satisfies the Einstein equation in the tangential directions
modulo $\dot{C}^\infty(\Z)$.
Then we can easily show
recursively that the right-hand sides of Equations \eqref{4.6} and \eqref{4.9} vanish modulo $\dot{C}^\infty(\Z)$. Thus $g^\Z$ is Einstein modulo $\dot{C}^\infty(\Z)$.

\begin{prop}\label{prop.formal.solution}
Let $(M^n,g)$ be a smooth Riemannian
manifold and let $W$ be a smooth symmetric field of endomorphisms of
$TM$ satisfying \eqref{e1} and \eqref{e2}. Then there exists on
$\Z:=(-\e,\e)\times M$ a metric $g^\Z$ of the form $g^\Z=dt^2+g_t$, with $g_0=g$,
whose Weingarten tensor at $t=0$ is $W$, and such that
\[\Ric^{\Z}-\lambda g^\Z\in\dot{C}^\infty(\Z).\]
Moreover, $g^\Z$ is unique up to $\dot{C}^\infty(\Z)$.
\end{prop}

\subsection*{A counterexample to long-time existence}
The preceding case of dimension $2+1$ hints that in general the Einstein metric 
$g^\Z$ cannot be extended on a complete manifold containing $M$ as a hypersurface. This sort of question is rather 
different from the arguments of this paper so we will only give an counterexample 
in dimension $1+1$ where global existence for the solution to the Cauchy problem fails. 
We restrict ourselves to the case of Ricci-flat metrics, which means vanishing Gaussian curvature in this dimension.

\begin{example}\label{exng}
Let $\Z$ be the incomplete flat surface obtained from $\cz^*$ (or from the complement of a small disk in $\cz$) by the following cut-and-paste procedure: cut along the positive real axis, then glue again after a translation of length $l>0$. More precisely, $x_+$ is identified with $(x+l)_-$ for all $x>\e$.
The resulting surface $\Z$ is clearly smooth and has a smooth flat metric including along the gluing locus. The unit circle in $\rz^2$ gives rise to a curve in~$\Z$ of curvature $1$ and length $2\pi$ with different endpoints $1_-$ and $(1+l)_-$. In a complete flat surface, a curve of curvature $1$ and length $2\pi$ must be closed (in fact smooth, since its lift to the universal cover must be a circle).
Therefore, the surface $\Z$ cannot be embedded in any complete flat surface. In particular, for any closed curve in $\Z$ circling around the singular locus, the interior cannot be continued to a compact flat surface with boundary.
\end{example}

\section{Spinors on Ricci-flat manifolds}

We come now to parallel and generalized Killing spinors, our main object of interest in this paper.
We keep the notation from the previous section. Our starting point is
the following corollary of Theorem \ref{rf}:

\begin{cor} \label{srf} Assume that $(M^n,g)$ is an analytic spin 
  manifold carrying a non-trivial generalized Killing spinor $\p$ 
  with analytic stress-energy tensor $W$. Then in a neighborhood
  of $\{0\}\times M$ in $\Z:=\rz \times M$ there exists a unique Ricci-flat
  metric $g^\Z$ of the form $g^\Z=dt^2+g_t$
  whose Weingarten tensor at $t=0$ is $W$.
\end{cor}
\begin{proof}
We just need to check that the constraints \eqref{e1}, \eqref{e2} are a consequence of \eqref{gks}.
In order to simplify the computations, we will drop the reference to the metric $g$ 
and denote respectively by $\n$, $R$, $\Ric$ and $\Scal$ the Levi-Civita covariant derivative, 
curvature tensor, Ricci tensor and 
scalar curvature of $(M,g)$. As usual, $\{e_i\}$ will denote a local $g$-orthonormal basis
of $TM$. 

We will use the following two classical formulas in Clifford calculus. The first one 
is the fact that 
the Clifford contraction of a symmetric tensor $A$ only depends on its trace:
\be\label{trace} \sum_{i=1}^n e_i\.A(e_i)=-\tr(A).
\ee
The second formula expresses the Clifford contraction of the spin curvature in terms 
of the Ricci tensor (\cite{bfgk}, p. 16):
\begin{align}\label{ric0} \sum_{i=1}^n
e_i\. R_{X,e_i}\p=-\tfrac12\Ric(X)\.\p,&&\forall X\in TM,\ \forall \p\in\Sigma M.
\end{align}

Let now $\p$ be a non-trivial generalized Killing spinor satisfying \eqref{gks}. 
Being parallel with respect to a modified connection on $\Sigma M$,
$\p$ is nowhere vanishing (and actually of constant norm).

Taking a further covariant derivative in \eqref{gks} and skew-symmetrizing yields
\[
 R_{X,Y}\p=\tfrac14\left(W(Y)\.W(X)-W(X)\.W(Y)\right)\.\p+
\tfrac12\left((\n_XW)(Y)-(\n_YW)(X)\right)\.\p
\]
for all $X,Y\in TM$. In this formula we set $Y=e_i$, take the Clifford product with $e_i$ and 
sum over $i$. From \eqref{trace} and \eqref{ric0} we get
\begin{align*}
 \Ric(X)\.\p=&-\tfrac12\sum_{i=1}^n e_i\.\left(W(e_i)\.W(X)-W(X)\.W(e_i)\right)\.\p\\
&-\sum_{i=1}^n e_i\.
\left((\n_XW)(e_i)-(\n_{e_i}W)(X)\right)\.\p\\
=&\tfrac12\tr(W)W(X)\.\p+\tfrac12\sum_{i=1}^n\big(-W(X)\.e_i-2g(W(X),e_i)\big)\.W(e_i)\.\p\\
&+\n_X(\tr(W))\p+\sum_{i=1}^n e_i\.(\n_{e_i}W)(X)\.\p.
\end{align*}
whence
\be\label{ric1}
\Ric(X)\.\p=\tr(W)W(X)\.\p-W^2(X)\.\p+X(\tr(W))\p+\sum_{i=1}^n e_i\.(\n_{e_i}W)(X)\.\p.
\ee
We set $X=e_j$ in \eqref{ric1}, take the Clifford product 
with $e_j$ and sum over $j$. Using \eqref{trace} again we obtain
\begin{align*}
-\Scal\, \p&=-\tr^2(W)\p+\tr(W^2)\p+\n(\tr(W))\.\p+\sum_{i,j=1}^n e_j\.e_i\.(\n_{e_i}W)(e_j)\.\p\\
&=-\tr^2(W)\p+\tr(W^2)\p+d\tr(W)\.\p+\sum_{i,j=1}^n (-e_i\.e_j-2\delta_{ij})\.(\n_{e_i}W)(e_j)\.\p\\
&=-\tr^2(W)\p+\tr(W^2)\p+2d\tr(W)\.\p+2\delta W\.\p,
\end{align*}
which implies simultaneously
\eqref{e1} and \eqref{e2} (indeed, if $f\p=X\.\p$ for some real $f$ and vector
$X$, then $-|X|^2\p=X\.X\.\p=X\.(f\p)=f^2\p$, so both $f$ and $X$ vanish).
\end{proof}

\begin{theorem} \label{ps} Let $(\Z,\gz)$ be a Ricci-flat spin
  manifold with Levi-Civita connection $\nz$ and let $M\subset\Z$ be
  any oriented analytic hypersurface. Assume there 
  exists some spinor $\p\in C^\infty(\Sigma\Z\res_M)$ which is
  parallel along $M$:
\begin{align}
\label{hyp}\nz_X\p=0,&&\forall X\in TM\subset T\Z.
\end{align}
Assume moreover that the application $\pi_1(M)\to\pi_1(\Z)$
induced by the inclusion is onto.
Then there exists a parallel spinor $\P\in C^\infty(\Sigma \Z)$ such
that $\P\res_M=\p$.
\end{theorem}
\begin{proof}
Any Ricci-flat manifold is analytic, cf.\ \cite{kdt}, \cite{besse}, thus the
analyticity of $M$ makes sense. The proof is split in two parts.

\subsubsection*{Local extension}
Let $\nu$ denote the unit normal vector field along $M$. 
Every $x\in M$ has an open neighborhood $V$ in $M$ such that the
exponential map $(-\e,\e)\times V\to \Z$, 
$(t,y)\mapsto \exp_y(t\nu)$ is well-defined for some $\e>0$. Its
differential at $(0,x)$ being the identity, one can assume, by
shrinking $V$ and choosing a smaller $\e$ if necessary, that it maps
$(-\e,\e)\times V$ diffeomorphically onto some open neighborhood $U$
of $x$ in $\Z$. We extend the spinor $\p$ to a spinor $\P$ on $U$ by
parallel transport 
along the normal geodesics $\exp_y(t\nu)$ for every fixed $y$. It remains to
prove that $\P$ is parallel on $U$ in horizontal directions.

Let $\{e_i\}$ be a local orthonormal basis along $M$. We extend it on
$U$ by parallel transport along the normal geodesics, and notice that
$\{e_i,\nu\}$ is a local orthonormal basis on $U$. More generally,
every vector field $X$ along $V$ gives rise to a unique horizontal vector field, also
denoted $X$, on $U$ satisfying $\n_\nu X=0$. For every such vector
field we get
\be\label{f1}\nz_\nu(\nz_X\P)=R^\Z(\nu,X)\P+\nz_{[\nu,X]}\P=
R^\Z(\nu,X)\P+\nz_{W(X)}\P.\ee 

Since $\Z$ is Ricci-flat, \eqref{ric0} applied to the local orthonormal basis
$\{e_i,\nu\}$ of $\Z$ yields
\be\label{ric} 0=\tfrac12\Ric^\Z(X)\.\P=\sum_{i=1}^n
e_i\. R^\Z(e_i,X)\P+\nu\. R^\Z(\nu,X)\P\,. 
\ee
We take the Clifford product with $\nu$ in this relation,
differentiate again with
respect to $\nu$ and use the second Bianchi identity to obtain:
\begin{align*} \nz_\nu(R^\Z(\nu,X)\P)=&\nz_\nu\left(\nu\.\sum_{i=1}^n e_i\.
R^\Z(e_i,X)\P\right)=\nu\.\sum_{i=1}^n e_i\. (\nz_\nu R^\Z)(e_i,X)\P \\ =& 
\nu\.\sum_{i=1}^n e_i\. \big((\nz_{e_i} R^\Z)(\nu,X)\P+(\nz_X
R^\Z)(e_i,\nu)\P\big),
\end{align*}
whence
\be\label{nr}
\begin{split}\nz_\nu(R^\Z(\nu,X)\P)=\nu\.\sum_{i=1}^n e_i\. &\left( \nz_{e_i}
(R^\Z(\nu,X)\P)+R^\Z(W(e_i),X)\P
-R^\Z(\nu,\nz_{e_i}X)\P \right. \\
&-R^\Z(\nu,X)\nz_{e_i}\P
+\nz_X(R^\Z(e_i,\nu)\P)-R^\Z(\nz_Xe_i,\nu)\P\\
&\left. +R^\Z(e_i,W(X))\P-R^\Z(e_i,\nu)\nz_X\P\right).
\end{split}
\ee
Let $\nu ^\perp$ denote the distribution orthogonal to $\nu$ on $U$ and
consider the sections $A, B\in C^\infty((\nu ^\perp)^*\otimes\Sigma U)$
and $C\in C^\infty(\Lambda^2(\nu ^\perp)^*\otimes\Sigma U)$ defined
for all $X,Y\in\nu ^\perp$ by
\begin{align*}
A(X):=\nz_X\P,&& B(X):=R^\Z(\nu,X)\P,&&
C(X,Y):=R^\Z(X,Y)\P.
\end{align*}
We have noted that the metric $g^\Z$ is analytic since it is Ricci-flat.
From the assumption that $M$ is analytic and that
$\p$ is parallel along $M$ it follows that $\P$, and thus the tensors
$A$, $B$ and $C$, are analytic.

Equations \eqref{f1} and \eqref{nr} read in our new notation:
\be\label{f3}(\nz_\nu A)(X)=B(X)+A(W(X)),\ee
and
\be\label{f2}\begin{split} (\nz_\nu B)(X)=\nu\.\sum_{i=1}^n
e_i\. &\big((\nz_{e_i}B)(X)+C(W(e_i),X)
-R^\Z(\nu,X)A(e_i)\\
&-(\nz_XB)(e_i)+C(e_i,W(X))-R^\Z(e_i,\nu)A(X)\big).\end{split}
\ee
Moreover, the second Bianchi identity yields
\begin{align*} (\nz_\nu C)(X,Y)=&(\nz_\nu
  R^\Z)(X,Y)\P=(\nz_X R^\Z)(\nu,Y)\P+(\nz_Y 
R^\Z)(X,\nu)\P \\
=&\nz_X
(R^\Z(\nu,Y)\P)-R^\Z(\nz_X\nu,Y)\P-R^\Z(\nu,\nz_XY)\P-R^\Z(\nu,Y)\nz_X\P\\
&-\nz_Y
(R^\Z(\nu,X)\P)+R^\Z(\nz_Y\nu,X)\P+R^\Z(\nu,\nz_YX)\P+R^\Z(\nu,X)\nz_Y\P\\
=&(\nz_X B)(Y)+C(W(X),Y)-R^\Z(\nu,Y)\nz_X\P\\
&-(\nz_Y B)(X)+C(X,W(Y))+R^\Z(\nu,X)\nz_Y\P,
\end{align*}
thus showing that
\be\label{f4}\begin{split} (\nz_\nu C)(X,Y)=&(\nz_X
  B)(Y)+C(W(X),Y)-R^\Z(\nu,Y)(A(X))\\ 
&-(\nz_Y B)(X)+C(X,W(Y))+R^\Z(\nu,X)(A(Y)).\end{split}
\ee
The hypothesis \eqref{hyp} is equivalent to $A=0$ for
$t=0$. Differentiating this again in the direction of $M$ and
skew-symmetrizing yields $C=0$ for $t=0$. Finally, \eqref{ric} shows
that $B=0$ for $t=0$. We thus see that the section
$S:=(A,B,C)$ vanishes on along the hypersurface $\{0\}\times V$ of $U$.

The system \eqref{f2}--\eqref{f4} is a linear PDE for $S$
and the hypersurfaces $t={\rm constant}$ are clearly non-characteristic.
The Cauchy-Kowalewskaya theorem shows that $S$ vanishes
everywhere on $U$. In particular, $A=0$ on $U$, thus proving our claim.

%%%%%%%%%%%%%%%%%%%%%%%%%%%%%%%%%%%%%%%%%%%%%%%%%%%%%%%
\subsubsection*{Global extension}
%%%%%%%%%%%%%%%%%%%%%%%%%%%%%%%%%%%%%%%%%%%%%%%%%%%%%%

Now we prove that there exists a parallel spinor $\P\in
  C^\infty(\Sigma\Z)$ such that $\P\res_M=\p$.
Take any $x\in M$ and an open neighborhood $U$ like in Theorem \ref{ps}
on which a parallel spinor $\P$ extending $\p$ is defined. The
spin holonomy group $\widetilde\Hol(U,x)$ thus preserves $\P_x$.
Since any Ricci-flat metric is analytic (cf. \cite[p.\ 145]{besse}), the
restricted spin holonomy group $\widetilde\Hol_0(\Z,x)$ is equal to
$\widetilde\Hol_0(U,x)$ for every $x\in\Z$ and for every open neighborhood
$U$ of $x$.
By the local extension result proved above, $\widetilde\Hol_0(U,x)$ acts 
trivially on $\P_x$, thus showing that $\P_x$ can be extended (by parallel transport along every
curve in $\tilde{\Z}$ starting from $x$) to a parallel spinor $\tilde\P$
on the universal cover 
$\tilde\Z$ of $\Z$. The deck transformation group acts trivially
on $\tilde{\P}$ since every element in $\pi_1(\Z,x)$ can be represented by a curve in $M$
(here we use the surjectivity hypothesis) and $\P$ was assumed to be parallel along $M$.
Thus $\tilde\P$ descends to $\Z$ as a parallel spinor.
\end{proof}

This result, together with Corollary \ref{srf} yields the solution to the analytic
Cauchy problem for parallel spinors stated in Theorem \ref{m1}.

%%%%%%%%%%%%%%%%%%%%%%%%%%%%%%%%%%%%%%%%%%%%%%%%%%%%%%
\section{Construction of generalized Killing spinors} \label{counterex}
%%%%%%%%%%%%%%%%%%%%%%%%%%%%%%%%%%%%%%%%%%%%%%%%%%%%%%

The goal of this section is to describe a method which yields 
generalized Killing spinors on many $3$-dimensional spin Riemannian manifolds.
We will obtain both analytic and non-analytic generalized Killing spinors. 
The analytic ones will yield examples for applying Theorem~\ref{m1}.
The non-analytic ones only yield a formal Taylor series in the sense of Proposition 
\ref{prop.formal.solution},
and we will show that in general this solution is not the Taylor series of a Ricci-flat metric.
Thus we see that the analyticity assumption in Theorem~\ref{m1} 
cannot be removed. The method consists in combining techniques 
developed elsewhere. We state below the relevant results and briefly explain the underlying ideas. 

Note that further examples of manifolds with generalized Killing spinors
which can not be embedded as hypersurfaces in manifolds with parallel spinors
were recently constructed (although not explicitly stated), by Bryant \cite{bryant10} in the context of $K$-structures
satisfying the so-called {\em weaker torsion condition}.

%%%%%%%%%%%%%%%%%%%%%%%%%%%%%%%%%%%%%%%%%%%%%%%%%%%%%%%%%%%%%%%%
\subsection{Minimizing the first Dirac eigenvalue in a conformal class}
%%%%%%%%%%%%%%%%%%%%%%%%%%%%%%%%%%%%%%%%%%%%%%%%%%%%%%%%%%%%%%

In \cite{ammann:09} and \cite{ammann:habil} the following problem was studied:
Suppose $M$ is an $n$-dimensional compact spin manifold, $n\geq 2$ endowed with a fixed spin structure.
For any metric $g$ on $M$ let $D^g$ be the Dirac operator on $M$. The spectrum 
of $D^g$ is discrete, and all eigenvalues have finite multiplicity. 
The first positive eigenvalue of $D^g$ will be denoted by $\la_1^+(g)$. In general, the dimension
of the kernel of $D^g$ depends on $g$, and on many manifolds (in particular on all compact spin manifolds 
of dimension $n\equiv 0,1,3,7$ $\mod$ $8$, $n\geq 3$) 
metrics $g_i$ are known such that 
$g_i\to g$ in the $C^\infty$-topology, $\dime \kernel D^{g_i}< \dime \kernel D^g$ and $\la_1^+(g_i)\to 0$. 
Thus $g\mapsto \la_1^+(g)$ is not continuous when defined on the set of all metrics.

We now fix a conformal class $[g_0]$ on~$M$, 
and only consider metrics $g\in [g_0]$. Then the above
properties change essentially.
Due to the conformal behavior of the Dirac operator, the dimension of the kernel of $D^g$ is constant on $[g_0]$, 
and furthermore $[g_0]\to \mR_+$, $g\mapsto \la_1^+(g)$ is continuous in the $C^1$-topology. 
For any positive real number $\a$ one has
$\la_1^+(\al^2g)= \al^{-1}\la_1^+(g)$. The normalized first positive eigenvalue function
$[g_0]\to (0,\infty)$, $g\mapsto \la_1^+(g)\vol(M,g)^{1/n}$, is thus scaling invariant and continuous
in the $C^1$-topology. It is unbounded from above, see \cite{ammann.jammes:11},
and bounded from below 
by a positive constant,  
see \cite{lott:86} in the case $\kernel D^{g_0}=0$ and 
\cite{ammann:03,ammann:habil} for the general case.
We introduce 
\be\label{def.inf}
  \lamin(M,[g_0]):=\inf_{g\in [g_0]} \la_1^+(g)\vol(M,g)^{1/n}>0.
\ee

If there is a metric of positive scalar curvature in $[g_0]$, then the Yamabe 
constant 
\be\label{def.yamabe}
  Y(M,[g_0]):=\inf_{g\in[g_0]} \vol(M,g)^{(2-n)/n}\int_M \Scal^g\,dv^g
\ee
is positive, and 
Hijazi's inequality \cite{hijazi:86,hijazi:91} then yields 
\be\label{hijazi}
 \lamin(M,[g_0])^2 \geq \tfrac{n}{4 (n-1)}Y(M,[g_0]).
\ee
\begin{example} If $(M,g_0)=\mS^n$ is the sphere $S^n$ with the conformal structure given by the standard
metric $\si^n$ of volume $\omega_n$, then Obata's theorem \cite[Prop.~6.2]{obata:71.72} 
implies that
the infimum in \eref{def.yamabe} is attained in $g=\sigma$, and 
thus
$Y(\mS^n)=n(n-1)\,\om_n^{2/n}$. We obtain 
$\lamin(\mS^n)\geq \tfrac{n}{2}\om_n^{1/n}$. On the other hand $(M,\si)$ 
carries a Killing spinor to the Killing
constant $-1/2$, thus $\la_1^+(\si)=\tfrac{n}{2}$. As a consequence, equality is 
attained
in~\eref{hijazi}, the infimum in~\eref{def.inf} is attained in 
$g=\si$ and $\lamin(\mS^n)= \tfrac{n}{2}\om_n^{1/n}$.
\end{example}

Now let $(M,g_0)$ be again arbitrary. By ``blowing up a sphere''
one can show that $\lamin(M,[g_0])\leq \lamin(\mS^n)$,
see \cite{ammann:03,ammann.grosjean.humbert.morel:08}. This inequality 
should be seen as a spinorial analogue of Aubin's inequality between the Yamabe 
constants $Y(M,[g_0])\leq Y(\mS^n)=n(n-1)\,\om_n^{2/n}$. 
For the Yamabe constants one even gets a stronger statement: If $(M,g_0)$ is 
not conformal to the round sphere, then
\be\label{ineq.yam.strict}
   Y(M,[g_0])<Y(\mS^n).
\ee
This inequality leads to a solution of the Yamabe problem, see 
\cite{lee.parker:87}. It was proved in some cases by
Aubin \cite{aubin:76}. Later Schoen and Yau
\cite{schoen:84,schoen.yau:88} 
could solve the remaining cases, using the positive mass theorem.

It is thus natural to ask the following question which is still open in general.
\begin{question}
Under the assumption that $(M,[g_0])$ is not conformal to $(\mS^n)$, $n\geq 2$, does 
the inequality
\be\label{ineq.strict}
   \lamin(M,[g_0])<\lamin(\mS^n).
\ee
always hold?
\end{question}

We will explain
below that many Riemannian manifolds, in particular ``generic'' metrics on 
compact spin $3$-dimensional manifolds, do satisfy \eref{ineq.strict}.
It is interesting to notice that using~\eref{hijazi} the 
inequality~\eref{ineq.strict} would imply~\eref{ineq.yam.strict} without 
referring to the positive mass theorem.

In analogy to the Yamabe problem 
which consists in finding a smooth metric attaining the infimum 
in~\eref{def.yamabe}, one can try to find a metric attaining the infimum 
in~\eref{def.inf}. If this infimum is achieved in a metric $g\in [g_0]$,
then the corresponding Euler-Lagrange equation provides the existence
of an eigenspinor $\psi$ of constant length of eigenvalue $\la_1^+(g_0)$.
In dimension $n=3$, such constant-length eigenspinors
are generalized Killing spinors, see 
Subsection~\ref{subsec.get.gen.Kil}, and -- as said above -- it is the 
goal of this section to construct generalized Killing spinors. 

Unfortunately, it is unclear whether the infimum in ~\eref{def.inf}
can be achieved by a (smooth) 
metric. However, if we assume that ~\eref{ineq.strict} holds,
and if we allow degenerations in the conformal factor, the infimum
is attained. To explain the nature of these possible degenerations
precisely, we introduce the following. A generalized metric in the conformal 
class $[g_0]$ is a metric of the form $f^2g_0$ 
where $f$ is continuous on $M$ and smooth on $M^*:=M\setminus f^{-1}(0)$. 
Moreover, we only admit such generalized metrics for which $M^*$ is dense 
in $M$. 
The set of all such admissible generalized metric associated 
to the conformal class $[g_0]$ will be denoted by $\overline{[g_0]}$.

\begin{remark}
The above definitions are slight more restrictive than in~\cite{ammann:09}, 
but sufficient for the purpose of the present article 
and didactically simpler. For example,
the condition that~$M^*$ is supposed to be dense, guarantees 
that $\overline{[g_0]}\cap \overline{[g_1]}=\emptyset$ if $g_0$ and 
$g_1$ are not conformal.
\end{remark}

The functions $\la_+^1,\vol:[g_0]\to \mR^+$ extend continuously to functions
$\overline{[g_0]}\to \mR^+$, and the infimum in~\eref{def.inf} does not 
change when we replace $[g_0]$ by $\overline{[g_0]}$.
We then have
\begin{theorem}[{\cite[Theorem~1.1(B)]{ammann:09}}]\label{theo.minim}
Let $(M,g_0)$ be a compact Riemannian spin manifold of dimension $n\geq 2$. 
There exists a generalized metric $g\in \overline{[g_0]}$ at which 
the infimum in~\eref{def.inf} is attained. On $(M^*,g)$ there 
exists a spinor $\psi$ of constant length with $D\psi = \la^1_+(g) \psi$.
\end{theorem}

The key idea in the proof of this theorem is to reformulate
the problem of minimizing ~\eref{def.inf} as a variational problem. For this we define
\begin{align}\label{def.F}
  \cF_q(\phi)=  \frac{\int\<D^{g_0}\phi ,\phi\>_{g_0}\dvol^{g_0}} {\|D^{g_0}\phi\|_{L^q(g_0)}^2}, &&
     \mu_q^{g_0}:= \sup \cF_q^{g_0}(\phi),
\end{align}
where the supremum runs over all spinors $\phi$ of regularity $C^1$ which are not in the kernel of $D^{g_0}$.
It was shown in \cite[Prop.~2.3]{ammann:09}
that for $q=\tfrac{2n}{n+1}$ we have 
  $$\mu_{2n/(n+1)}^{g_0}=\tfrac{1}{\lamin(M,[g_0])}.$$
Furthermore the infimum in \eqref{def.inf} is attained in a smooth metric $g\in[g_0]$ if and only if 
there is a nowhere vanishing spinor $\psi_0$ which attains the supremum in~\eqref{def.F}. 
If the infimum is attained in $g$ and the supremum
in $\psi_0$, then both are related via 
\begin{equation}\label{spinormetric}
g=|D^{g_0}\psi_0|^{4/(n+1)}g_0.
\end{equation}

\begin{prop}[{\cite[Theorem~1.1~(A)]{ammann:09}}]\label{prop.sup.attain}
Under the condition~\eqref{ineq.strict} the supremum is attained in a spinor $\psi_0$ of regularity $C^{2,\alpha}$ for small $\al>0$.
\end{prop}

The strategy of proof is similar to the classical approach to the 
Yamabe problem as e.g.\ in \cite{lee.parker:87}.
A maximizing sequence for the functional will in general not converge, 
due to conformal invariance.
One then defines ``perturbed'' or ``regularized'' modifications of this
functional such that their maximizing sequences converge to a maximizer. 
In a final step one shows, assuming~\eref{ineq.strict}, 
that the maximizers of the perturbed functionals converge to a 
maximizer of the unperturbed functional.

Let us now continue with the sketch of proof of Theorem~\ref{theo.minim}.
From Prop.\ \ref{prop.sup.attain} we know that the supremum of $\cF$ is attained at some spinor $\psi_0$ 
which satisfies an Euler-Lagrange equation. By suitably rescaling $\psi_0$ and by possibly adding an element of $\ker D^{g_0}$
to $\psi_0$, the Euler-Lagrange equation reads
\begin{align*}
D^{g_0}\psi_0=\lamin(M,[g_0]) |\psi_0|^{2/(n-1)}\psi_0,&& \|\psi_0\|_{L^{2n/(n-1)}(g_0)}=1.
\end{align*}
However, it is unclear whether $D^{g_0}\psi_0$ (or equivalently $\psi_0$) has zeros or not, 
and therefore if the metric $g$ defined in \eqref{spinormetric} makes sense.

We will show in the following subsection that the zero set is nowhere dense, in other 
words its complement is dense. Then $g:=|D^{g_0}\psi_0|^{4/(n+1)}g_0$ defines a generalized metric, and by naturally extending 
the definition of $\la_1^+$ to generalized metrics, we see that the infimum in \eqref{def.inf} is then attained 
in this generalized metric.

Consistently with the above we set $M^*:= M\setminus \psi_0^{-1}(0)$.
From the standard formula for the behavior of the Dirac operator under conformal change 
(see e.g.\ \cite{hijazi:01}) the spinor $\psi:=\tfrac{\psi_0}{|\psi_0|}$ on 
$M^*$ satisfies
\begin{align*}
D^g\psi=\lamin(M,[g_0]) \psi, && |\psi|\equiv 1.
\end{align*}
This finishes the proof for Theorem~\ref{theo.minim}, up to the density of $M^*$ explained below.

%%%%%%%%%%%%%%%%%%%%%%%%%%%%%%%%%%%%%%%%%%%%%%%%%%%%%%%%%%%%%%%%%%%%%%%%%%%%%%5
\subsection{The zero set of the maximizing spinor}
%%%%%%%%%%%%%%%%%%%%%%%%%%%%%%%%%%%%%%%%%%%%%%%%%%%%%%%%%%%%%%%%%%%%%%%%%%

The goal of this subsection is to study the zero set  
of the maximizing spinor $\psi_0$ from the previous section.

\begin{lemma} \label{dense}
Let $(M,g_0)$ be a connected Riemannian spin manifold. 
Assume that a spinor $\phi$ of regularity $C^1$ satisfies
\begin{equation}\label{eq.non-lin.r}
D^{g_0}\phi=c|\phi|^r\phi
\end{equation}
where $r\geq 0$ and $c\in \mR$. If $\phi$ vanishes on a non-empty open set, then it vanishes on $M$.
\end{lemma}

Applying the lemma to $\phi:=\psi_0\not\equiv 0$ and $r:=2/(n-1)$ one obtains the density of~$M^*$ in $M$.
\begin{proof}
The lemma is a special case of the weak unique continuation principle~\cite{booss.marcolli.wang:02}.
More exactly we apply \cite[Theorem 2.7]{booss.marcolli.wang:02} with $\Dirac_A=D^{g_0}$ and $\fP_A(\phi,x):=-|\phi(x)|^r$.
As $\phi$ is locally bounded, we see that $x\mapsto \fP_A(\phi,x)$ is locally bounded as well.  Thus $\fP_A$ is an admissible perturbation
in the sense of  \cite{booss.marcolli.wang:02}, and \cite[Theorem 2.7]{booss.marcolli.wang:02} then yields the weak unique continuation principle 
for this equation which is exactly the statement of the lemma.
\end{proof}

We propose two conjectures around the above lemma.

The first conjecture relies on the following remark: 
if $r$ is an even integer, then $|\phi|^r\phi$ is a smooth function of $\phi$, so the Main Theorem in~\cite{baer:99} 
shows that the zero set of $\phi$ is a countably $(n-2)$-rectifiable set, and thus of Hausdorff dimension at most $n-2$.
In contrast, if $r$ is not an even integer, then B\"ar's method of proof does not apply, but the result seems likely to remain true.
\begin{conjecture}\label{conj1}
The zero set of any solution of~\eref{eq.non-lin.r} is of Hausdorff dimension at most $n-2$.
\end{conjecture}

The second conjecture is motivated from the following, cf.\
\cite{hermann:phd,hermann:12}: for generic metrics on a compact $2$- or $3$-dimensional 
spin manifold all eigenspinors, i.e. all non-trivial 
solutions of~\eqref{eq.non-lin.r} 
with $r=0$, do not vanish anywhere; in other words they are everywhere non-zero.

We conjecture that the same fact is true for $r:=\tfrac{2}{(n-1)}$. 
This constant $r$ is special, as then \eqref{eq.non-lin.r} and thus the zero set 
of $\phi$ is conformally invariant.

\begin{conjecture}
Let $r:=\tfrac{2}{(n-1)}$, and let $M$ be connected.
For generic conformal classes on $M$, any solution 
of \eqref{eq.non-lin.r} with $\phi\not\equiv 0$ is everywhere non-zero.
\end{conjecture}

If Conjecture~\ref{conj1} holds and if $M$ is connected, then 
the manifold $M\setminus \phi^{-1}(0)$ is connected. 
Fortunately, for the maximizing spinor $\psi_0$ the following fact can be proven independently of the above conjectures:

\begin{lemma}\label{connected}
Assume $M$ to be connected. Let $\psi_0$ be the maximizing spinor provided by 
Proposition~\ref{prop.sup.attain}.
Then $M^*=M\setminus \psi_0^{-1}(0)$ is connected.
\end{lemma}
\begin{proof}
Assume that there exists a partition $M^*=\Omega_1\sqcup \Omega_2$
into non-empty disjoint open sets. We define the continuous spinor
$\psi_1$ by $\psi_1|_{\Omega_1}:=\psi_0|_{\Omega_1}$ and 
$\psi_1|_{M\setminus \Omega_1}:\equiv 0$. Then 
$\|\psi_1\|_{L^{2n/(n-1)}}<\|\psi_0\|_{L^{2n/(n-1)}}$. As a first step we prove 
by contradiction that 
$\psi_1$ is $C^1$, or equivalently that $\na \psi_0=0$ on $\partial \Omega_1$.

Suppose that there existed $x\in \partial \Omega_1\cap \partial \Omega_2$ 
such that $\nabla\psi_0$
is non-zero in $x$. Because of $(D\phi)(x)=0$ the map $T_xM\to \Sigma_x M$, 
$X\mapsto \nabla_X \psi_0$ has rank at least $2$. The implicit function 
theorem then implies that there is a connected open neighborhood $U$ of $x$ 
and a submanifold $S\subset U$ of codimension $2$ such that 
$\psi_0^{-1}(0)\cap U\subset S$. This implies that $U\setminus S\subset \Omega_1$. One easily concludes that $S\cap \Omega_2=\emptyset$, thus we obtain the 
contradiction $x\not\in \partial\Omega_2$.

We have proven that $\psi_1$ is $C^1$, and thus $\psi_1$ is a solution to
\begin{align*}
 D^{g_0}\psi_1=\lamin(M,[g_0]) |\psi_1|^{2/(n-1)}\psi_1,&& 0< \|\psi_1\|_{L^{2n/(n-1)}(g_0)}<1.
\end{align*}

A straightforward calculation then yields 
  $$\cF_{2n/(n+1)}(\psi_1)>  \tfrac{1}{\lamin(M,[g_0])}=\mu_{2n/(n+1)}^{g_0}$$
which contradicts the definition of $\mu_{2n/(n+1)}^{g_0}$.
\end{proof}

%%%%%%%%%%%%%%%%%%%%%%%%%%%%%%%%%%%%%%%%%%%%%%%%%%%%%%
\subsection{From eigenspinors of constant length to generalized Killing spinors}
%%%%%%%%%%%%%%%%%%%%%%%%%%%%%%%%%%%%%%%%%%%%%%%%%%%%%
\label{subsec.get.gen.Kil}

In this section we specialize to the case  $n=3$. We will
see that in this dimension any eigenspinor of constant length is 
a generalized Killing spinor.

\begin{proposition}
Let $\psi$ be a solution of $D\psi=H\psi$, $H\in C^\infty(M)$, 
of constant length $1$, on a manifold of dimension $n=3$.
Then $\psi$ is a generalized Killing spinor. 
\end{proposition}

This proposition is the natural generalization of a result in
\cite{friedrich:98} from $n=2$ to $n=3$. 
We will include a simple proof here.
 
\begin{proof}
Let $g$ be the metric on~$M$ and $\langle\cdot,\cdot\rangle$ the real part of the Hermitian 
metric on~$\Sigma M$. We define $A\in \End(TM)$ by 
  $$g(A(X),Y):=\<\na_X\psi,Y\cdot \psi\>$$
for all $X,Y\in TM$.
Note that for any point $p\in M$ and any vector $X\in T_pM$ we have
$$\<\na_X\psi, \psi\>=\tfrac12\pa_{X}\<\psi,\psi\>=0,$$
in other words 
$\na_X \psi\in \psi^\perp=\{\phi\in \Si_pM\,|\, \<\phi,\psi\>=0\}$.
Let $e_1,e_2,e_3$ be an orthonormal basis of $T_pM$. By possibly changing the order of this basis, we can achieve  $e_1\cdot e_2 \cdot e_3=1$
in the sense of endomorphisms of $\Si M$. 
The spinors $e_1\cdot\psi$, $e_2\cdot\psi$ and $e_3\cdot \psi$ 
form an orthonormal system of $\psi^\perp$, and because of 
$\dim_\mR \psi^\perp=3$, it is a basis. It follows 
$\na_X\psi=A(X)\cdot \psi$.

Furthermore
\begin{eqnarray*} 
  \<A(e_2),e_1\> &=&  \< \na_{e_2} \psi , e_1\cdot \psi\>=
 \< e_2 \cdot \na_{e_2} \psi , \underbrace{e_2\cdot e_1\cdot\psi}_{e_3\cdot \psi}\>\\ 
&=& H \underbrace{\< \psi,e_3\cdot\psi\>}_{=0}  
    - \< e_1 \cdot \na_{e_1}\psi,e_3 \cdot\psi\>   
    - \< e_3\cdot \na_{e_3}\psi, e_3\cdot \psi\>\\ 
&=&  \< e_3 \cdot e_1 \cdot \na_{e_1}\psi,\psi\>   
    - \<\na_{e_3}\psi, \psi\>= - \< e_2 \cdot \na_{e_1}\psi,\psi\>\\
 &=& \<A(e_1),e_2\> 
\end{eqnarray*} 
and similarly  
$\<A(e_1),e_3\>= \<A(e_3),e_1\>$ and $\<A(e_2),e_3\>=\<A(e_3),e_2\>$. 
Thus $A$ is symmetric.
\end{proof}

Summarizing our knowledge until now, we have:
\begin{cor}
Assume that $(M,g_0)$ is a compact connected spin manifold of dimension $n=3$ 
satisfying $\lamin(M,[g_0])<\lamin(\mS^3)=\tfrac32\, (2\pi^2)^{1/3}$. 
Then there exist
\begin{enumerate}
\item an open, connected and dense subset $M^*$;
\item a metric $g$ on $M^*$ conformal to $g_0|_{M^*}$and of volume~$1$;
\item an eigenspinor $\psi$ to $D^g$ to the real eigenvalue $\la_+^1(g)$,
\end{enumerate}
such that $\psi$ has constant length and thus is a generalized Killing spinor
on $(M^*,g)$. We obtain a self-adjoint section $A$ of $\End(TM)$ such that 
$\na_X\psi=A(X)\cdot \psi$ and $\tr A=- \la_+^1(g)$.
\end{cor}

%%%%%%%%%%%%%%%%%%%%%%%%%%%%%%%%%%%%%%%%%%%%%%%%%%%%%%
\subsection{Analytic manifolds}\label{subsec.real-ana}
%%%%%%%%%%%%%%%%%%%%%%%%%%%%%%%%%%%%%%%%%%%%%%%%%%%%%%

\begin{definition}\label{def.conf.ana}
Let $g_1$ be a Riemannian metric on a smooth manifold $M$. We say that~$[g_1]$ is
an \emph{analytic conformal class} if $M$ has a compatible
structure of a (real-)analytic manifold for which 
one of the following equivalent statements holds:
\begin{enumerate}[{\rm (a)}]
\item there is a (real-)analytic metric $h\in [g_1]$;
\item for any point $x\in M$ there is an open set $U\ni x$, such that there is an analytic
metric $g^U$ on $U$ with $g^U\in [g_1|_U]$.
\end{enumerate}
\end{definition}

\begin{lemma}
Conditions (a) and (b) in Definition~\ref{def.conf.ana} are equivalent.
\end{lemma}
\begin{proof}
The implication from (a) to (b) is trivial. The implication from (b) to (a) is a 
direct consequence of uniformization in dimension $n=2$, thus we restrict to the case $n\geq 3$. 

Let $g$ be a smooth metric in the given analytic conformal class. We have to show that the locally defined 
metrics $g^U$ provided by (b) can be deformed conformally such that they match together to a 
globally defined metric.
Let $L_p:=\{\la g_p\,|\, \la>0\}\subset T_p^*M\otimes T_p^*M$, and let 
$L:=\bigcup L_p$. The bundle $\pi:L\to M$ is a smooth $\mR^+$-principal bundle
over $M$. All local Riemannian metrics $g^U$ are 
local sections of $\pi:L\to M$, $g^U:U\to L$. 
If two local analytic metrics
$g^U$ and $g^{\ti U}$ are given, then there is an analytic function
$f:\ti U\cap U\to \mR^+$ such that $g^U=f g^{\ti U}$ on $\ti U\cap U$.
Consequently, $\pi:L\to M$ carries a structure of
analytic $\mR^+$-principal bundle over $M$, and thus the total space $L$
of the bundle is an analytic manifold.
The smooth map $g:M\to L$ can be approximated in the strong $C^1$-topology by 
an analytic map $g^\omega:M\to L$, see \cite[Chap.~2, Theorem~5.1]{hirsch:76} which is 
proven by Grauert and Remmert in \cite{grauert:58}. 
The map $\pi\circ g^\omega:M\to M$ 
is a smooth analytic map, which is close to the identity in the $C^1$-topology, 
and thus (for suitably chosen $g^\omega$) it is an analytic diffeomorphism. 

As a consequence, the map $g^\omega\circ (\pi\circ g^\omega)^{-1}: M\to L$ is an analytic 
section of $L$ and thus an analytic representative of the given conformal class.
\end{proof}

\begin{lemma}\label{lemma.conf.ext}
If an analytic conformal class is conformally flat on a non-empty open set~$U$, and if~$M$ is connected, then the conformal class is 
already conformally flat on $M$.
\end{lemma}

\begin{proof}
Being conformally flat on an open set $U$ is equivalent to the vanishing of the Weyl curvature (resp.~Schouten tensor) in dimension
$m\geq 4$ (resp. $m=3$). The Weyl curvature and the Schouten tensor of an analytic metric are analytic as well. 
Thus if they vanish on~$U$ they must vanish on all of~$M$.
\end{proof}

\begin{lemma}
Let $\phi$ be a smooth solution of $D\phi= c|\phi|^\alpha \phi$, $\phi\neq 0$ 
on a (not necessarily complete)
analytic Riemannian spin manifold $(U,g,\chi)$. 
Then $\phi$ is analytic as well.
\end{lemma}

\begin{proof}
The equation is an elliptic semi-linear equation, and has analytic 
coefficients on the set~$M\setminus\phi^{-1}(0)$. 
We apply analytic regularity results for properly elliptic systems 
as developed by Douglis and Nirenberg and refined by Morrey, 
see~\cite{morrey:58} and~\cite{morrey:66}. To apply these tools 
it is convenient to deduce a second order equation
  $$D^2 \phi = \left(c^2 |\phi|^{2\alpha} + c \grad(|\phi|^\alpha)\right)\cdot \phi$$
which has again analytic coefficients on~$M\setminus\phi^{-1}(0)$. 
The linearization of this second order equation has the principal symbol 
of a Laplacian and is thus properly elliptic.
The lemma then follows directly from \cite[Theorem~6.8.1]{morrey:66} 
or~\cite{morrey:58}.
\end{proof}

\begin{lemma}\label{cmc.analytic}
Constant mean curvature hypersurfaces in an analytic 
Riemannian manifolds are analytic. In particular,
the metric and the second fundamental form of such a hypersurface are analytic.
\end{lemma}

\begin{proof}
Let $M$ be an $n$-dimensional hypersurface in an analytic Riemannian manifold
$(\cZ,h)$ of dimension $n+1$. We choose an analytic parametrization
$ U\times (a,b)\to \cZ$ with $U$ open in~$\mR^n$, such that locally the hypersurface~$M$ 
is the graph of a function $F:U\to (a,b)$. 
The standard basis of $\mR^{n+1}$ is denoted 
by $e_1,\ldots, e_{n+1}$.
The tangent space $T_{(x,F(x))} M$ is then spanned
by $(e_i, \pa_i F)$, $i=1,\ldots, n$.

Let $h_{ij}\in C^\om(U\times (a,b))$ be the coefficients of the metric $h$, and let $g_{ij}\in C^\infty(U)$
be the coefficients of $g$. The inverse matrices are denoted by $(h^{ij})_{1\le i,j\le n+1}$ and $(g^{ij})_{1\le i,j\le n}$.

The first fundamental form of the hypersurface in the chart given by $U$ is 
  $$g_{ij}=h_{ij}+ h_{n+1,j} \pa_i F +h_{n+1,i}\pa_j F  
    + h_{n+1,n+1} (\pa_i F)(\pa_j F).$$
The coefficients of the matrices $(g_{ij})$ and $(g^{ij})$ are
thus polynomial expressions in $h$, $F$ and~$d F$.
The vector field
 $$X:=\sum_{i=1}^{n+1} \Bigl(\Bigl( -\sum_{j=1}^n h^{ij}\pa_j F\Bigr) + h^{i,n+1}\Bigr)e_i$$
is normal to $M$, and both $X$ and the unit normal vector field
$\nu:=X/|X|_h$ are analytic expressions in $h$, $F$ and $d F$.

 The second fundamental form has the coefficients
\begin{eqnarray*}
k_{ij}&=& -\< \na_{(e_i,\pa_i F)}\nu, (e_j,\pa_j F)\>\\
      &=& -\tfrac1{|X|_h} \< \na_{(e_i,\pa_i F)}X, (e_j,\pa_j F)\>\\
      &=&  \tfrac1{|X|_h}\left(\pa_i\pa_j F + \cF(h,dh,F,dF)\right),
\end{eqnarray*}
where $1\le i,j\le n$ and $\cF$ is a polynomial expression in its arguments.

The mean curvature $H$ is given as $H=\tfrac1n \sum_{ij}g^{ij} k_{ij}$.
Thus the mean curvature operator $P:F\mapsto H$ is a quasi-linear second order
differential operator with analytic coefficients. 

We fix a function $\ti F$ describing a hypersurface of constant mean curvature, 
the corresponding normal field will be denoted by $\ti X$. In other words
$P(\ti F)$ is a constant. 

The linearization $\hat P:= T_{\ti F}P$ of $P$ in $\ti F$ 
is a linear second order differential operator  with principal symbol
\begin{align*}
\mR^m\to \mR,&&\xi\mapsto \tfrac{|\xi|^2}{|\ti X|_h}.
\end{align*}
Thus $P$ is (properly) elliptic in a neighborhood of $0$. 

The analytic regularity theorem for elliptic systems of Morrey \cite[Theorem~6.8.1]{morrey:66} or~\cite{morrey:58} tells us 
that $\ti F$ is analytic, and this implies the lemma.
\end{proof}

 %%%%%%%%%%%%%%%%%%%%%%%%%%%%%%%%%%%%%%%%%%%%%%%%%%%%%%
\subsection{Three-dimensional real projective space}
%%%%%%%%%%%%%%%%%%%%%%%%%%%%%%%%%%%%%%%%%%%%%%%%%%%%%%

In this and in the following subsection we provide examples of 
compact Riemannian spin manifolds satisfying~\eref{ineq.strict}. 
In the present subsection we study deformations of round metrics 
on $\mR P^3$ with a suitable spin structure. This already 
provides examples of non-analytic Riemannian manifolds with 
generalized Killing spinor, showing the necessity of the analyticity 
assumption in Theorems~\ref{m1} and~\ref{rf}. In the following section we 
will then see that such examples are abundant.

\begin{lemma}If $M$ is a compact spin manifold, we denote the set of 
metrics with invertible Dirac operator as $\cRi(M)$, equipped with the $C^1$-topology. Then the function 
\begin{align*}
\cRi(M)\to \mR^+,&& g\mapsto \lambda_1^+(g)
\end{align*}
is continuous.
\end{lemma}

This lemma is a special case of \cite[Prop.~7.1]{baer:96b}, see also \cite[Kor.~1.3.3]{pfaefflediss} for more details.

Let us equip $\SU(2)$ with the unique bi-invariant
metric of sectional curvature $1$, hence $\SU(2)$ is isometric to $\mS^3$. The left multiplication of 
$\SU(2)$ on itself lifts to an action of $\SU(2)$ on $\Si \SU(2)$, for any choice of orientation of $\SU(2)$ 
and any choice of the spinor representation. The spinor bundle is then trivialized by left-invariant spinors.
A straightforward calculation, 
see e.g.\ \cite{ammann:habil}, 
shows that 
  $$\na_X \phi =\pm \tfrac12 X\cdot \phi$$
for any left-invariant spinor $\phi$ and all $X\in T\SU(2)$. Thus all left-invariant spinors are Killing spinors
to the Killing constant $\pm 1/2$. The sign depends on the choice of orientation and on the choice of spinor representation.
The same discussion also applies to right-invariant spinors, and these are Killing spinors whose Killing constant have the opposite sign.
We assume that these choices are 
made such that left-invariant spinors have Killing constant $-1/2$, and thus right-invariant ones have Killing constant $+1/2$.

If $\Gamma$ is a non-trivial discrete subgroup of $\SU(2)$, we choose a spin structure on $\Gamma\backslash \SU(2)$ such that left-invariant spinors 
on $\mS^3$ descend to $\Gamma\backslash \SU(2)$. 
Then $\Gamma\backslash \SU(2)$ carries a complex $2$-dimensional space of
Killing spinors with Killing constant $-1/2$, but no non-trivial Killing spinor with Killing constant $1/2$. For quotients $\SU(2)/\Gamma$,
the role of $1/2$ and $-1/2$ have to be exchanged. All other (Riemannian) quotients of $\mS^3$ 
do not carry any non-trivial Killing spinor.

In the special case $\Gamma=\{\pm\one\}$ both quotients $\Gamma\backslash \SU(2)$ and $\SU(2)/\Gamma$ are isometric to $\mR P^3$, 
but they come with different spin structures. These are the two non-equivalent spin structures on $\mR P^3$.
We thus have obtained:

\begin{lemma}
Let $\si^3$ be the standard metric on the $3$-dimensional real projective space
$\mR P^3$. There are two spin structures on $\mR P^3$. For one spin structure
Killing spinors to the constant $-1/2$ exist, but not for the constant $1/2$.
For the other spin structure Killing spinors to the constant $1/2$ exist, 
but not for the constant $-1/2$.
\end{lemma}

Thus for a suitable choice of spin structure, we have 
  $$\lamin(\mR P^3,[\si^3])=\tfrac32 \left(\tfrac{\om_3}2\right)^{1/3}=\tfrac{3\pi^{2/3}}2<\lamin(\mS^3)=\tfrac32 \om_3^{1/3}=\tfrac{3\pi^{2/3}}{2^{2/3}}.$$

\begin{cor}
There is a non-analytic conformal class and a spin structure
on $\mR P^3$ for which inequality~\eref{ineq.strict} holds.
\end{cor}

\begin{proof}
We choose a metric $g_1$ close to $\si^3$ on $\mR P^3$ which is conformally flat on 
some non-empty open set $U_1$ and non-conformally flat 
on some non-empty open set $U_2$. If $g$ were an analytic metric, conformal
to $g_1$, then $g$ would have a vanishing Schouten tensor on $U_1$. 
By Lemma \ref{lemma.conf.ext} it would be flat everywhere, thus also on $U_2$. 
This shows that $[g_1]$ is a non-analytic conformal class.
Applying the previous lemmata, we obtain the corollary.
\end{proof}

%%%%%%%%%%%%%%%%%%%%%%%%%%%%%%%%%%%%%%%%%%%%%%%%%%%%%%%%%%%%%
\subsection{The mass endomorphism and application to inequality~\eref{ineq.strict}}
%%%%%%%%%%%%%%%%%%%%%%%%%%%%%%%%%%%%%%%%%%%%%%%%%%%%%%%%%%%%
 
The goal of this subsection is to prove that 
inequality~\eref{ineq.strict} holds for ``generic'' metrics, in a sense explained below.

In this section we assume that $M$ is a compact connected spin manifold of dimension $n\geq 3$, 
and that the index of $M$ in $KO^{-n}(pt)$ vanishes. We fix a point $p\in M$ and a flat metric $\gfl$ 
in a neighborhood $U$ of $p$, $U\neq M$. We assume that $U$ is isometric to a convex ball and that 
$\gfl$ can be extended to a metric on $M$. The set of all such extensions is denoted by $\cR_{U,\gfl}(M)$.
We define 
\begin{equation*}
  \cRinv(M):=
  \{ g \in \cR_{U,\gfl}(M) \,|\, \mbox{$D^g$ is invertible} \},
\end{equation*}
i.e.\ this is the set of all extensions of $\gfl$ such that the Dirac operator is invertible. 
In~\cite{ammann.dahl.humbert:p09} we proved that $\cRinv(M)$ is open and dense in $\cR_{U,\gfl}(M)$ 
with respect to the $C^k$-topology where $k\geq 1$ is arbitrary.

\begin{definition}\label{def.generic}
We say that a property (A) holds for generic metrics in  $\cR_{U,\gfl}(M)$ if there is a subset $\cR'\subset \cR_{U,\gfl}(M)$
that is open and dense with respect to the $C^k$-topology for all $k\geq 1$, such that property (A) holds for all $g\in \cR'$.
\end{definition}

Using this definition, the above mentioned result from \cite{ammann.dahl.humbert:p09} says that the Dirac operator
with respect to a generic metric is invertible.

Given a metric $g\in\cRinv(M)$, let $G$ be the Green's function of the Dirac operator on $(M,g)$ at the point $p\in M$, i.e.\
a distributional solution of 
\begin{equation}\label{def.green}
  D G= \delta_p \Id_{\Sigma_p M},
\end{equation}
where $\delta_p$ is the Dirac distribution at $p$ and $G$ is viewed
as a linear map which associates to each spinor in $\Sigma_p M$ a
smooth spinor field on $M \setminus \{p\}$ defining a spinor-valued distribution on $M$. 
We write $G^g$ and $D^g$ for $G$ and $D$ to indicate their dependence on the metric $g$.

We also introduce the Euclidean Green's function centered at $0$, defined distributionally on $\mathbb{R}^n$
\begin{equation*}
  G^{\rm eucl} \psi = - \tfrac{1}{\omega_{n-1} |x|^n} x \cdot \psi.  
\end{equation*}
It satisfies \eref{def.green} for $G=G^{\rm eucl}$ and $D=D^{\rm eucl}$ on $\mathbb{R}^n$.

Identifying $U$ with a ball in $\mathbb{R}^n$ via an isometry, 
both $G=G^g$ and $G=G^{\rm eucl}$ are solutions of \eref{def.green} on $U$. Thus $D^g(G^g-G^{\rm eucl})=0$ on $U$ and by elliptic regularity, $G^g-G^{\rm eucl}$ is a smooth section, see also \cite{ammann.humbert.morel:06}.
We obtain for any $\psi_0 \in \Sigma_p M$: 
  \begin{equation*}
    G^g (x) \psi_0 
    =  
    - \tfrac{1}{\omega_{n-1}|x|^n} x \cdot \psi_0
    + v^g(x) \psi_0,
  \end{equation*}
where the spinor field $v^g(x) \psi_0$ is smooth on $U$ and satisfies $D^g ( v^g(x)  \psi_0 ) = 0$ 
on $U$.

\begin{definition}\label{def.mass}
  The {\it mass endomorphism} $\alpha^g: \Sigma_p M \to \Sigma_p M$
  for a point $p \in U \subset M$ is defined by
  \begin{equation*}
    \alpha^g (\psi_0) := v^g(p) \psi_0 .
  \end{equation*}
\end{definition}
The mass endomorphism is thus (up to a constant) defined as the zero$^{\rm{th}}$
order term in the asymptotic expansion of the Green's function in
Euclidean coordinates around $p$. 
This definition is analogous to the definition of the mass in the Yamabe problem.

\begin{theorem}[\cite{hermann:10} for $n=3$, \cite{ammann.dahl.hermann.humbert:p10a} for $n\geq 3$]
For generic metrics in  $\cR_{U,\gfl}(M)$ the mass endomorphism in $p$ is non-zero.
\end{theorem}

An important application of this theorem is inequality~\eref{ineq.strict}. 
The proofs in \cite{ammann.humbert.morel:06} yield:
\begin{prop}
If the mass endomorphism in a point $p$ with flat neighborhood is non-zero, then
$\lamin(M,[g])<\lamin(\mS^n)$.
\end{prop}

It follows:
\begin{cor}
For generic metrics $g$ in  $\cR_{U,\gfl}(M)$ we have $\lamin(M,[g])<\lamin(\mS^n)$.
\end{cor}

We now deduce:
\begin{cor}\label{cor.exist.an.non-an}
Let $M$ be an $n$-dimensional compact spin manifold with vanishing index $\ind(M)\in KO^{-n}(pt)$.
There there is both an analytic conformal class $[g_{\rm an}]$ and a non-analytic, smooth conformal class $[g_{\rm non-an}]$
on $M$ with 
\begin{align*}
\lamin(M,[g_{\rm an}])<\lamin(\mS^n),&& \lamin(M,[g_{\rm non-an}])<\lamin(\mS^n).
\end{align*}
\end{cor}

In this corollary $M$ is a priori equipped with a $C^\infty$-structure and the ``non-analyticity'' means by definition that $M$ 
does not carry any analytic structure in which $g_{\rm non-an}$ is analytic.

\begin{proof}
We choose an open set $U$ and a metric $\gfl$ as above. Then choose $g\in \cR_{U,\gfl}(M)$ with $\lamin(M,[g])<\lamin(\mS^n)$.
Choose another smooth metric $g_{\rm non-an}$, coinciding on $U$ with $g=\gfl$, such that $g_{\rm non-an}$ 
is not (everywhere) conformally flat on $M\setminus U$, and $C^1$-close enough to $g$ so that $\lamin(M,[g_{\rm non-an}])<\lamin(\mS^n)$. 
The metric $g_{\rm non-an}$ is conformally flat on $U$ but not on $M\setminus U$, hence its Schouten tensor cannot be analytic 
in any analytic structure. Thus as in Lemma~\ref{lemma.conf.ext} the conformal class $[g_{\rm non-an}]$
cannot be analytic.

At the same time, $g$ can be $C^1$-approximated by an analytic metric $g_{\rm an}$ so that the inequality 
$\lamin(M,[g_{\rm an}])<\lamin(\mS^n)$ continues to hold.
Such an analytic approximation can be done either with Abresch's smoothing technique
or by using the Ricci flow: if $g_t$ is a solution of the Ricci flow equation $\tfrac{d}{dt}(g_t)= -2 \Ric^{g_t}$, defined for short times
$t\in [0,t_0)$ with initial data $g_0=g$,
then $g_t$ is analytic for all $t>0$. We set $g_{\rm an}:=g_t$ for a sufficiently small $t>0$.
\end{proof}

%%%%%%%%%%%%%%%%%%%%%%%%%%%%%%%%%%%%%%%%%%%%%%%%%%%%%%%%%%%%
\subsection{Analytic examples}
%%%%%%%%%%%%%%%%%%%%%%%%%%%%%%%%%%%%%%%%%%%%%%%%%%%%%%%%%%%%%
Summarizing the results of the preceding subsections we obtain.
\begin{theorem}
Assume that $(M,[g_{\rm an}])$ is a compact connected analytic Riemannian spin manifold of dimension $3$ with 
$\lamin(M,[g_{\rm an}])<\lamin(\mS^n)$. Then there is a connected, open and dense subset $M^*$ of $M$ carrying an analytic metric~$g^*$
and an analytic spinor field $\psi\in \Gamma(\Sigma^{g^*} M^*)$ such that
\begin{enumerate}
 \item $g^*$ is conformal to~$g_{\rm an}|_{M^*}$;
 \item $\vol(M^*,g^*)=1$;
 \item $\psi$ is a generalized-Killing spinor on $(M^*,g^*)$.
\end{enumerate}
\end{theorem}

Such Riemannian metrics $g_{\rm an}$ exist on each compact $3$-dimensional spin manifold, due to the preceding section.
The corresponding endomorphism $W$ is then analytic as well, and Theorem~\ref{m1} can be applied.
We obtain a Ricci-flat metric of the form $dt^2+g_t$ where $g_0=g^*$ defined on an open neighborhood of 
$\{0\}\times M^*$ in $\mR\times M^*$, and carrying a parallel spinor. Further the mean curvature
of $\{0\}\times M^*$ in this neighborhood is constant and equal to $(2/3)\lamin(M,[g])$.

%%%%%%%%%%%%%%%%%%%%%%%%%%%%%%%%%%%%%%%%%%%%%%%%%%%%%%%%%%%%%
\subsection{Non-analytic examples}\label{non.analytic.non.exist}
%%%%%%%%%%%%%%%%%%%%%%%%%%%%%%%%%%%%%%%%%%%%%%%%%%%%%%%%%%%%%

Here we finally prove the existence of metrics with generalized Killing spinors on manifolds with non-analytic metrics. According to Lemma \ref{cmc.analytic} such manifolds do not embed isometrically as constant mean curvature hypersurfaces in Ricci-flat manifolds. Thus the analyticity 
assumptions in Theorem~\ref{m1} cannot be removed.

\begin{theorem}\label{n-a}
Any $3$-dimensional compact connected spin manifold $M$ with a fixed $C^\infty$-structure has a connected open dense subset $M^*$ 
carrying a smooth Riemannian metric~$g^*$ with a generalized Killing spinor, such that the metric is not analytic 
for any choice of analytic structure on $M$. For this manifold $(M^*,g^*)$ 
the associated formal solution provided by Proposition~\ref{prop.formal.solution} cannot be chosen to be Ricci-flat on a neighborhood
of $\{0\}\times M$, in other words the conclusion of Theorem~\ref{m1} does not hold.
If $M=\mR P^3$ or more generally if $M=\Gamma\backslash \SU(2)$ 
for a non-trivial subgroup $\Gamma$ of $\SU(2)$, then we can find such a 
Riemannian metric $g^*$ defined on the whole manifold $M$.
\end{theorem}

\begin{proof}
By Corollary~\ref{cor.exist.an.non-an} there exists a smooth 
conformal class $[g_{\rm non-an}]$ whose Schouten tensor 
vanishes on a non-empty open set and does not vanish on another 
open set, and for which $\lamin(M,[g_{\rm non-an}])<\lamin(\mS^3)$.
The infimum in \eqref{def.inf} is then attained, according to 
Theorem~\ref{theo.minim}, at a generalized metric $g^*$, 
which is a smooth Riemannian metric on a connected dense 
open subset $M^*$ of $M$. It is clear that the restricted 
conformal class $[g_{\rm non-an}|_{M^*}]$ is not analytic, 
and thus the metric $g^*$ cannot be analytic either. 
Theorem~\ref{theo.minim} provides moreover a Dirac eigenspinor 
of constant length on $(M^*,g^*)$ which is, due to 
Subsection~\ref{subsec.get.gen.Kil}, a generalized Killing spinor. 
Furthermore, the trace of the associated symmetric tensor 
$W\in \End(TM)$ is constant and equal to $-(2/3)\lamin(M,[g_{\rm non-an}])$. 

If the formal solution provided by Proposition~\ref{prop.formal.solution} (for $M^*$
instead of $M$) were
Ricci-flat in a neighborhood of $\{0\}\times M^*$, then $M^*$ would be a hypersurface of constant mean curvature in a $4$-dimensional
Ricci-flat manifold. As Ricci-flat metrics are analytic in a suitable analytic structure, Lemma~\ref{cmc.analytic} 
would imply that $g^*$ was analytic, which is a contradiction.

Now assume that $M=\mR P^3$. We take a sequence of non-analytic
metrics $g_i$ (constructed similarly as above) converging in the $C^\infty$-topology to the standard round metric $\si^3$ on $\mR P^3$. 
As the functional $\cF_q^g$ depends continuously on $g$ in the $C^1$-topology, 
we see for $q=2n/(n+1)$
\begin{equation}\label{ineq.RP3}
\overline{\mu}:=\liminf_{i\to \infty} \mu_q(\mR P^3,g_i)\geq  \mu_q(\mR P^3,\si^3)=\tfrac23 \left(\tfrac2{\om_3}\right)^{2/3}> \mu_q(\mS^3)=\tfrac23 \left(\tfrac1{\om_3}\right)^{2/3}
\end{equation}

Now let $\psi_i$ be a maximizing spinor on $(\mR P^3,g_i)$ with 
$L^p$-norm $1$, $p=2n/(n-1)=3$. These $\psi_i$ are uniformly bounded
in the $C^0$-norm. This uniform $C^0$-boundedness follows from 
\cite[Theorem~6.1]{ammann:09} whose proof is also valid 
for $p=2n/(n-1)$ although the formulation of \cite[Theorem~6.1]{ammann:09} 
assumed $p<2n/(n-1)$. Then \cite[Theorem~5.2]{ammann:09} implies that 
$\psi_i$ is a bounded sequence in $C^{1,\alpha}$ for any $\alpha\in(0,1)$.

After passing to a suitable subsequence we then see that
$\psi_i$ converges to a solution $\overline{\psi}$ of 
\begin{align*}
D^{\si^3}\overline\psi = \overline{\mu}^{-1}|\overline\psi| \overline\psi,
&& \|\overline\psi\|_{L^3(\mR P^3,\si^3)}=1.
\end{align*}

Calculating $\cF_{2n/(n+1)}(\overline\psi)=\bar\mu$ we conclude 
$\overline{\mu}=\mu_q(\mR P^3,\si^3)$.
Using the regularity theorem \cite[Prop.~5.1]{ammann:09} one sees that $\overline\psi$ is $C^2$. We now apply \cite[Prop.~4.1]{ammann:09} where 
$\psi\in\Gamma(\Sigma \mS^3)$ is the pullback of $\overline\psi$ to $\mS^3$.
One calculates $\cF_{2n/(n+1)}(\psi)=\mu_q(\mS^3)$, thus
$\psi$ is a maximizing spinor on $\mS^3$.
The conformal map $A:\mS^3\to \mS^3$ in the conclusion 
of  \cite[Prop.~4.1]{ammann:09}
has to be an isometry as it is the lift of a map $\mR P^3\to \mR P^3$.
Thus  \cite[Prop.~4.1]{ammann:09} implies that $\overline\psi$ 
is a Killing spinor to the Killing constant $-1/2$.
As such a Killing spinor nowhere vanishes, $\psi_i$ nowhere vanishes 
for large $i$. 
 
The other quotients $\Gamma\backslash \SU(2)$ are completely analogous. 
\end{proof}

Using products with manifolds carrying parallel spinors, one can easily obtain in every dimension $n\ge 3$ examples of $n$-dimensional manifolds 
with generalized Killing spinors which do not embed isometrically as hypersurfaces in manifolds with parallel spinors. More precisely we have
the following:

\begin{lemma}\label{product}
Let $(M^*,g^*)$ be a 3-dimensional non-analytic Riemannian manifold with generalized Killing spinors given by Theorem \ref{n-a}.
Then the Riemannian product $$(M^*,g^*)\times (\mR^{n-3},g_{\mathrm{eucl}})$$ carries a generalized Killing spinor $\Psi$ but can not be embedded
isometrically as a hypersurface in any manifold with parallel spinors which restrict to $\Psi$. 
\end{lemma}

\begin{proof}
Let $p_1^*(\Sigma M^*)$ and $p_2^*(\Sigma \mR^{n-3})$ denote the pullbacks to $Z:=M^*\times \mR^{n-3}$
of the spin bundles of $(M^*,g^*)$ and $(\mR^{n-3},g_{\mathrm{eucl}})$ with respect to the standard projections. It is a standard fact
that the spin bundle $\Sigma Z$ is isomorphic to $p_1^*(\Sigma M^*)\otimes p_2^*(\Sigma \mR^{n-3})$ if $n$ is odd
and to  $p_1^*(\Sigma M^*)\otimes p_2^*(\Sigma \mR^{n-3})\otimes \mC^2$ if $n$ is even, and this isomorphism preserves the 
spin connections. The isomorphism can be chosen such that in the first case, the Clifford product is given by 
$$(X_1,X_2)\.(\phi\otimes\psi)=(X_1\.\phi)\otimes\psi+\phi\otimes(X_2\.\omega^\mC\.\psi),$$
where $\omega^\mC$ is the complex volume form in the Clifford algebra of $\mR^{n-3}$.
In the second case, the Clifford product is given by 
$$(X_1,X_2)\.(\phi\otimes\psi\otimes v)=(X_1\.\phi)\otimes\psi\otimes a(v)+\phi\otimes(X_2\.\psi)\otimes b(v),$$
for every $v\in\mC^2$, where $a=\begin{pmatrix}1&0\\0&-1\end{pmatrix}$ and  $b=\begin{pmatrix}0&1\\1&0\end{pmatrix}$. 
The first assertion now follows immediately: take any generalized Killing spinor $\phi$ on $M^*$ 
satisfying $\nabla_X\phi=W(X)\.\phi$ for all $X\in T M^*$ and let $\psi$ be a parallel spinor on
$\mR^{n-3}$. One can of course assume that $\omega^\mC\.\psi=\psi$ if $n$ is odd. Then 
$\Psi:=\phi\otimes\psi$ (resp.\ $\Psi:=\phi\otimes\psi\otimes\begin{pmatrix}1\\0\end{pmatrix}$) is a generalized Killing spinor on $Z$
for $n$ odd (resp.\ even), with associated tensor $\bar W=\begin{pmatrix}W&0\\0&0\end{pmatrix}$.

To prove the second assertion, assume that $Z$ is a hypersurface in some spin manifold $\bar Z$ and that $\Phi$ is
a parallel spinor on $\bar Z$ restricting to $\Psi$ on $Z$. The second fundamental form of $Z$ is $\bar W$, which has
constant trace by construction. Thus $Z$ has constant mean curvature, so is analytic by Lemma~\ref{cmc.analytic}. 
Each factor of $Z$ is then analytic, contradicting the non-analyticity of $M^*$.
\end{proof}

%%%%%%%%%%%%%%%%%%%%%%%%%%%%%%%%%%%%%%%%%%%%%%%%%%%%%%%%%%%%%%%%%%%%%%%%%
\bibliographystyle{amsplain}
%%%%%%%%%%%%%%%%%%%%%%%%%%%%%%%%%%%%%%%%%%%%%%%%%%%%%%%%%%%%%%%%%%%%%%%%%

\bibliographystyle{amsplain}

\end{document}